\documentclass{article}
\usepackage{latexsym,a4,calc,makeidx}
\usepackage{stmaryrd,amsthm,amssymb,amsmath}
\usepackage{mathptmx,mathrsfs,newcent}
\usepackage{pstricks,pst-node,multido,pst-plot,graphicx,xypic}
\usepackage{times}

\newtheorem{thm}{Theorem}[section]
\newtheorem{lemma}[thm]{Lemma}
\newtheorem{prop}[thm]{Proposition}

\newtheorem{remark}[thm]{Remark}
\newtheorem{remarks}[thm]{Remarks}
\newtheorem{definition}[thm]{Definition}
\newtheorem{notation}[thm]{Notation}

\newtheorem{example}[thm]{Example}
\newtheorem{examples}[thm]{Examples}

\newcommand{\Remark}[1]{\begin{remark}{\rm #1}\end{remark}}

\newcommand{\Definition}[1]{\begin{definition}{\rm #1}\end{definition}}
\newcommand{\Notation}[1]{\begin{notation}{\rm #1}\end{notation}}
\newcommand{\Example}[1]{\begin{example}{\rm #1}\end{example}}
\newcommand{\Examples}[1]{\begin{examples}{\rm #1}\end{examples}}

\renewcommand{\Box}{\square}

\newcommand{\N}{{\mathbb N}}
\newcommand{\Co}{{\mathbb C}}
\newcommand{\R}{{\mathbb R}}
\newcommand{\K}{{\mathbb K}}

\newcommand{\DD}{\mathcal{D}}

\newcommand{\id}{{\rm id}}

\DeclareMathAlphabet{\mathpzc}{OT1}{pzc}{m}{it}

\newcommand{\grad}{\operatorname{grad}}

\newcommand{\Hom}{\operatorname{Hom}}

\newcommand{\tr}{\operatorname{tr}}

\newcommand{\dt}{{\,\,{dt}}}
\newcommand{\dV}{{\,\,\operatorname{dV}}}

\renewcommand{\div}{\operatorname{div}}

\newcommand{\supp}{\operatorname{supp}}
\newcommand{\ssupp}{\operatorname{sing\,supp}}

\newcommand{\la}{\langle}       
\newcommand{\ra}{\rangle}

\newcommand{\Csc}{C_{\mathrm{sc}}^\infty}

\renewcommand{\phi}{\varphi}

\newcommand{\lra}{\longrightarrow}

\def \be{\begin{eqnarray*}}
\def \ee{\end{eqnarray*}}
\def \ben{\begin{enumerate}}
\def \een{\end{enumerate}}
\def \beit{\begin{itemize}}
\def \eeit{\end{itemize}}
\def \bui#1#2{\mathrel{\mathop{\kern 0pt#1}\limits^{#2}}}
\def \buil#1#2{\mathrel{\mathop{\kern 0pt#1}\limits_{#2}}}

\newcounter{Abbildung}
\newcommand{\abb}[1]{\stepcounter{Abbildung}\begin{center}{\footnotesize
      Fig.\ \theAbbildung: #1}\end{center}}
\title{Linear wave equations on Lorentzian manifolds} 

\author{Christian B\"ar}

\date{\today}

\makeindex

\begin{document}

\maketitle

\begin{abstract}
We summarize the analytic theory of linear wave equations on globally
hyperbolic Lorentzian manifolds.
\end{abstract}

\tableofcontents

\section{Introduction}

In General Relativity \index{general relativity} spacetime \index{spacetime} is modelled by a Lorentzian manifold, \index{Lorentzian manifold}
see e.~g.\ \cite{EH,Wa}.
Many physical phenomena, such as electro-magnetic
radiation,\index{electro-magnetic radiation} are described by solutions to certain linear wave
equations defined on this spacetime manifold.
Thus a good understanding of the theory of wave equations is crucial.
This includes initial value problems (the Cauchy problem), fundamental
solutions, and inverse operators (Green's operators).
The classical textbooks on partial differential equations contain the relevant
results for small domains in Lorentzian manifolds or for very special
manifolds such as Minkowski space.

In this text we summarize the global analytic results obtained in \cite{BGP},
see also Leray's unpublished lecture notes \cite{Leray} and Choquet-Bruhat's
exposition \cite{C-B}.
In order to obtain a good solution theory one has to impose certain geometric
conditions on the underlying manifold.
The situation is similar to the study of elliptic operators \index{elliptic operator}on Riemannian
manifolds.\index{Riemannian manifold}
In order to ensure that the Laplace-Beltrami operator \index{Laplace-Beltrami operator}on a Riemannian manifold
$M$ is essentially self-adjoint \index{essentially self-adjoint}one may make the natural assumption that $M$
be complete.
Unfortunately, there is no good notion of completeness \index{completeness}for Lorentzian
manifolds.
It will turn out that the analysis of wave operators works out nicely if one
assumes that the underlying Lorentzian manifold be globally hyperbolic. 
Completeness of Riemannian manifolds and global hyperbolicity of Lorentzian
manifolds are indeed related.
If $(S,g_0)$ is a Riemannian manifold, then the Lorentzian cylinder \index{Lorentzian cylinder}$M=\R\times
S$ with product metric $g= -dt^2 + g_0$ is globally hyperbolic if and only if
$(S,g_0)$ is complete.

We will start by collecting some material on distributional sections in vector
bundles.
Then we will summarize the theory of globally hyperbolic Lorentzian manifolds.
Then we will define wave operators, also called normally hyperbolic operators,
and give some examples.
After that we consider the basic initial value problem, the Cauchy problem.
It turns out that on a globally hyperbolic manifold solutions exist and are
unique.
They depend continuously on the initial data.
The support of the solutions can be controlled which is physically nothing
than the statement that a wave can never propagate faster than with the speed
of light.
In the subsequent section we use the results on the Cauchy problem to show
existence and uniqueness of fundamental solutions.
This is closely related to the existence and uniqueness of Green's operators.

The author is very grateful for many helpful discussions with colleagues
including 
Helga Baum,
Olaf M\"uller,
Nicolas Ginoux,
Frank Pf\"affle,
and Miguel S\'anchez.
The author also thanks the Deutsche Forschungsgemeinschaft for financial
support.

\section{Distributional sections in vector bundles}
Let us start by giving some definitions and by fixing the
terminology for distributions on manifolds.
We will confine ourselves to those facts that we will actually need later on.
A systematic and much more complete introduction may be found e.~g.\ in
\cite{FL1}.

\subsection{Preliminaries on distributional sections}

Let $M$ be a manifold equipped with a smooth volume density $\dV$. 
Later on we will use the volume density induced by a Lorentzian metric
but this is irrelevant for now.
We consider a real or complex vector bundle $E\to M$.
We will always write $\K=\R$ or $\K=\Co$ depending on whether $E$ is real or
complex. 
The space of compactly supported smooth sections in $E$ will be denoted by
$\mathcal{D}(M,E)$. \index{compactly supported smooth sections}
We equip $E$ and the cotangent bundle $T^*M$ with connections, both denoted by
$\nabla$. \index{connection}
They induce connections on the tensor bundles $T^*M\otimes\cdots\otimes T^*M
\otimes E$, again denoted by $\nabla$.
For a continuously differentiable section $\phi\in C^1(M,E)$ the covariant
derivative is a continuous section in $T^*M\otimes E$, $\nabla\phi\in
C^0(M,T^*M\otimes E)$.
More generally, for $\varphi\in C^k(M,E)$ we get $\nabla^k\phi\in
C^0(M,\underbrace{T^*M\otimes\cdots\otimes T^*M}_{k\,\, \mathrm{factors}}\otimes E)$.

We choose an auxiliary Riemannian metric on $T^*M$ and an auxiliary Riemannian
or Hermitian metric on $E$ depending on whether $E$ is real or complex.
This induces metrics on all bundles $T^*M\otimes\cdots\otimes
T^*M\otimes E$.
Hence the norm of $\nabla^k\phi$ is defined at all points of $M$.

For a subset $A \subset M$ and $\varphi\in C^k(M,E)$ we define the $C^k$-norm
by
\begin{equation}
\|\varphi\|_{C^k(A)} := \max_{j=0,\ldots,k} \;
\sup_{x\in A}\,|\nabla^j\varphi(x)| .
\label{defCkNorm}
\end{equation}
If $A$ is compact, then different choices of the metrics and the
connections yield equivalent norms $\|\cdot\|_{C^k(A)}$.
For this reason there will be no need to explicitly specify the
metrics and the connections.

The elements of $\DD(M,E)$ are referred to as test sections \index{test sections}in $E$.
We define a notion of convergence of test sections.\index{test sections!convergence of}

\Definition{
Let $\varphi,\varphi_n\in \mathcal{D}(M,E)$. 
We say that the sequence $(\varphi_n)_n$ {\it converges to} $\varphi$ {\it in} 
$\mathcal{D}(M,E)$ if the following two conditions hold:
\begin{enumerate}
 \item There is a compact set $K\subset M$ such that the supports of $\varphi$ and of all
 $\varphi_n$ are contained in $K$, i.~e.\ $\supp(\varphi), \supp(\varphi_n)\subset K$ for 
 all $n$.
 \item The sequence $(\varphi_n)_n$ converges to $\varphi$ in all
 $C^k$-norms over $K$, i.~e.\ for each $k\in\N$
$$
\|\varphi-\varphi_n\|_{C^k(K)} \buil{\lra}{n\to\infty} 0.
$$
\end{enumerate}
}

We fix a finite-dimensional $\K$-vector space $W$.
Recall that $\K=\R$ or $\K=\Co$ depending on whether $E$ is real or complex.
Denote by $E^*$ the vector bundle over $M$ dual to $E$.

\Definition{\label{def:distr}
A $\K$-linear map $F:\mathcal{D}(M,E^*)\to W$ is called a 
{\it distribution in $E$ with values in $W$} or a {\it distributional section
  in $E$ with values in $W$} \index{distributional section}if it is continuous in the sense  
that for all convergent sequences $\varphi_n\to\varphi$ in
$\mathcal{D}(M,E^*)$  one has $ F[\varphi_n]\to F[\varphi]$.
We write $\DD'(M,E,W)$ for the space of all $W$-valued distributions in $E$. 
}

Note that since $W$ is finite-dimensional all norms $|\cdot|$
on $W$ yield the same topology on $W$.
Hence there is no need to specify a norm on $W$ for Definition~\ref{def:distr}
to make sense.
Note moreover, that distributional sections in $E$ act on test sections in $E^*$.

\Example{
Pick a  bundle $E\to M$ and a point $x\in M$.
The {\it delta-distribution}\index{delta-distribution} $\delta_x$ is a distributional section in $E$
with values in $E_x^*$. 
For $\varphi\in\mathcal{D}(M,E^*)$ it is defined by
\[ \delta_x[\varphi]=\varphi(x). \]
}

\Example{\label{ex:locint}
Every locally integrable section $f\in L^1_{\mathrm{loc}}(M,E)$ can be 
regarded as a $\K$-valued distribution in $E$ by setting for any 
$\varphi\in\mathcal{D}(M,E^*)$ 
\[ f[\varphi]:=\int_M \varphi(f) \dV.
\]
Here $\phi(f)$ denotes the $\K$-valued $L^1$-function with compact support on $M$
obtained by pointwise application of $\phi(x)\in E_x^*$ to $f(x)\in E_x$.
}

\subsection{Differential operators acting on distributions}
\label{subseq:DiffOpDist}

Let $E$ and $F$ be two $\K$-vector bundles over the manifold $M$, $\K=\R$ or
$\K=\Co$. 
Consider a linear differential operator $P:C^\infty(M,E)\to
C^\infty(M,F)$.\index{linear differential operator}
There is a unique linear differential operator $P^*:C^\infty(M,F^*)\to
C^\infty(M,E^*)$ called the {\em formal adjoint of} $P$\index{linear differential operator!formal adjoint of} such that
for any $\varphi\in \DD(M,E)$ and $\psi\in\DD(M,F^*)$
\begin{equation}
\int_M \psi(P\varphi)\, \dV =
\int_M (P^*\psi)(\varphi)\, \dV .
\label{formaladjoint}
\end{equation}
If $P$ is of order $k$, then so is $P^*$ and (\ref{formaladjoint}) holds for 
all $\varphi\in C^k(M,E)$ and $\psi\in C^k(M,F^*)$ such that
$\supp(\varphi)\cap\supp(\psi)$ is compact.
With respect to the canonical identification $E=(E^*)^*$ we have $(P^*)^*=P$.

Any linear differential operator $P:C^\infty(M,E)\to C^\infty(M,F)$ extends 
canonically to a linear operator $P:\DD'(M,E,W)\to\DD'(M,F,W)$\index{linear differential operator!acting on distributions} by
$$
(PT)[\varphi]:=T[P^*\varphi]
$$
where $\phi\in\DD(M,F^*)$.
If a sequence $(\varphi_n)_n$ converges in $\mathcal{D}(M,F^*)$ to $0$, 
then the sequence $(P^*\varphi_n)_n$ converges to $0$ as well 
because $P^*$ is a differential operator.
Hence $(PT)[\varphi_n] = T[P^*\varphi_n] \to 0$.
Therefore $PT$ is indeed again a distribution.

The map $P:\mathcal{D}'(M,E,W)\rightarrow \mathcal{D}'(M,F,W)$
is $\K$-linear.
If $P$ is of order $k$ and $\varphi$ is a $C^k$-section in $E$, seen as a 
$\K$-valued distribution in $E$, then the distribution $P\varphi$ coincides 
with the continuous section obtained by applying $P$ to $\varphi$ classically.

An important special case occurs when $P$ is of order $0$, i.~e.\ 
$P\in C^\infty(M,\Hom(E,F))$.
Then $P^*\in C^\infty(M,\Hom(F^*,E^*))$ is the pointwise adjoint.
In particular, for a function $f\in C^\infty(M,\K)$ we have 
$$
(fT)[\varphi] = T[f\varphi].
$$

\subsection{Supports}

\Definition{
The {\em support} \index{distributional section!support of}of a distribution $T\in\DD'(M,E,W)$ is defined as the set
\begin{eqnarray*}
\lefteqn{\supp(T)} \\
&:=& 
\{x\in M\,|\, \forall\mbox{ neighborhood $U$ of $x$ }\exists \,
\varphi\in \DD(M,E)\mbox{ with }\supp(\varphi)\subset U
\mbox{ and } T[\varphi]\not= 0\}.
\end{eqnarray*}
}
It follows from the definition that the support of $T$ is a closed subset of
$M$.
In case $T$ is a $L^1_{\mathrm{loc}}$-section this notion of support coincides
with the usual one for sections.

If for $\phi\in\DD(M,E^*)$ the supports of $\phi$ and $T$ are disjoint, then
$T[\phi]=0$.
Namely, for each $x\in\supp(\phi)$ there is a neighborhood $U$ of $x$ such
that $T[\psi]=0$ whenever $\supp(\psi)\subset U$.
Cover the compact set $\supp(\phi)$ by finitely many such open sets
$U_1,\ldots,U_k$. 
Using a partition of unity one can write $\phi=\psi_1 + \cdots + \psi_k$ with
$\psi_j\in\DD(M,E^*)$ and $\supp(\psi_j)\subset U_j$.
Hence 
$$
T[\phi]=T[\psi_1 + \cdots + \psi_k] = T[\psi_1] + \cdots + T[\psi_k]=0.
$$
Be aware that it is not sufficient to assume that $\phi$ vanishes on
$\supp(T)$ in order to ensure $T[\phi]=0$.
For example, if $M=\R$ and $E$ is the trivial $\K$-line bundle let
$T\in\DD'(\R,\K)$ be given by $T[\phi]=\phi'(0)$.
Then $\supp(T)=\{0\}$ but $T[\phi]=\phi'(0)$ may well be nonzero while
$\phi(0)=0$.

If $T\in\DD'(M,E,W)$ and $\varphi\in C^\infty(M,E^*)$, then the evaluation
$T[\varphi]$ can be defined if $\supp(T)\cap\supp(\varphi)$ is compact
even if the support of $\phi$ itself is noncompact.
To do this pick a function $\sigma\in \DD(M,\R)$ that is constant $1$ on
a neighborhood of $\supp(T)\cap\supp(\varphi)$ and put
$$
T[\varphi] := T[\sigma\varphi].
$$
This definition is independent of the choice of $\sigma$ since for another
choice $\sigma'$ we have 
$$
T[\sigma\varphi] - T[\sigma'\varphi] =
T[(\sigma-\sigma')\varphi] = 0
$$
because $\supp((\sigma-\sigma')\varphi)$ and $\supp(T)$ are disjoint.

Let $T\in\DD'(M,E,W)$ and let $\Omega\subset M$ be an open subset.
Each test section $\varphi\in\DD(\Omega,E^*)$ can be extended by $0$ and yields
a test section $\varphi \in \DD(M,E^*)$.
This defines an embedding $\DD(\Omega,E^*)\subset\DD(M,E^*)$.
By the restriction of $T$ to $\Omega$ we mean its restriction from 
$\DD(M,E^*)$ to $\DD(\Omega,E^*)$.

\Definition{
The {\em singular support}\index{distributional section!singular support of} $\ssupp(T)$ of a distribution $T\in\DD'(M,E,W)$ 
is the set of points which do not have a neighborhood restricted to which
$T$ coincides with a smooth section.
}

The singular support is also closed and we always have $\ssupp(T)
\subset\supp(T)$.

\Example{
For the delta-distribution $\delta_x$ we have
$\supp(\delta_x)=\ssupp(\delta_x)=\{x\}$.
}

\subsection{Convergence of distributions}

The space $\DD'(M,E)$ of distributions in $E$ will always be given the
{\em weak topology}.\index{weak topology}
This means that $T_n \to T$ in $\DD'(M,E,W)$ if and only if $T_n[\varphi]
\to T[\varphi]$ for all $\varphi\in\DD(M,E^*)$.
Linear differential operators $P$ are always continuous with respect to the
weak topology. 
Namely, if $T_n \to T$, then we have for every $\varphi\in\DD(M,E^*)$
$$
PT_n[\varphi] = T_n[P^*\varphi] \to T[P^*\varphi] = PT[\varphi].
$$
Hence
$$
PT_n \to PT.
$$

\Remark{
Let $T_n,T\in C^0(M,E)$ and suppose $\|T_n - T\|_{C^0(M)} \to 0$.
Consider $T_n$ and $T$ as distributions.
Then $T_n \to T$ in $\DD'(M,E)$.
In particular, for every linear differential operator $P$
we have $PT_n \to PT$.
}

\section{Globally hyperbolic Lorentzian manifolds}

Next we summarize some notions and facts from Lorentzian geometry.
More comprehensive introductions can be found in \cite{BEE} and in
\cite{ONeill}. 

By a {\em Lorentzian manifold}\index{Lorentzian manifold}  we mean a semi-Riemannian manifold whose
metric has signature $(-,+,\cdots,+)$.
We denote the Lorentzian metric \index{Lorentzian metric}by $g$ or by $\langle\cdot,\cdot\rangle$.
A tangent vector $X\in TM$ is called {\em timelike} \index{tangent vector!timelike} if $\langle X,X\rangle <
0$, {\em lightlike}\index{tangent vector!lightlike} if $\langle X,X\rangle = 0$ and $X\not=0$, {\em causal} \index{tangent vector!causal}if
it is timelike or lightlike, and {\em spacelike} \index{tangent vector!spacelike}otherwise.
At each point $p\in M$ the set of timelike vectors in $T_pM$ decomposes into
two connected components.
A {\em timeorientation} \index{timeorientation}on $M$ is a choice of one of the two
connected components of timelike vectors in $T_pM$ which depends continuously
on $p$.
This means that we can find a continuous timelike vector field on $M$ taking
values in the chosen connected components.
Tangent vectors in the chosen connected component are called {\em future
  directed},\index{tangent vector!future directed} those in the other component are called {\em past directed}.\index{tangent vector!past directed}

Let $M$ be a timeoriented Lorentzian manifold.
A piecewise $C^1$-curve in $M$ is called {\em timelike, lightlike, causal, 
spacelike, future directed}, \index{curve!timelike}\index{curve!lightlike}\index{curve!causal}\index{curve!spacelike}\index{curve!future directed}\index{curve!past directed}or {\em
past directed} if its tangent vectors are timelike, lightlike, causal,
spacelike, future directed, or past directed respectively.

The {\em chronological future}\index{chronological future} $I_+^M(x)$ of a point $x\in M$ is the set of
points that can be reached from $x$ by future directed timelike curves. 
Similarly, the {\em causal future}\index{causal future} $J_+^M(x)$ of a point $x\in M$ consists 
of those points that can be reached from $x$ by causal curves and of $x$
itself. 
The {\em chronological future} of a subset
$A\subset M$ is defined to be $I_+^M(A):=\buil{\cup}{x\in A}I_+^M(x)$. 
Similarly, the {\em causal future} of $A$ is $J_+^M(A):=\buil{\cup}{x\in
  A}J_+^M(x)$.
The {\em chronological past}\index{chronological past} $I_-^M(A)$ and the {\em causal past}\index{causal past} $J_-^M(A)$
are defined by replacing future directed curves by past directed curves.
One has in general that $I_\pm^M(A)$ is the interior of $J^M_\pm(A)$ and
that $J_\pm^M(A)$ is contained in the closure of $I_\pm^M(A)$.
The chronological future and past are open subsets but the causal
future and past are not always closed even if $A$ is closed.

\begin{center}

\begin{pspicture*}(-5,-3)(5,3)


\psset{unit=0.7cm}

\psccurve(-2,0)(0,-0.67)(2,0)(1,0.8)(-0.5,0.67)
\rput(0,0){$A$}

\psline[linecolor=blue](1.92,-0.2)(5,3)
\psline[linecolor=blue](-1.92,-0.25)(-5,3)
\rput(0,3.5){\blue \small $J_+^M(A)$}
\psline[linewidth=0.5pt]{->}(0,3.2)(0,2.5)
\psecurve[linewidth=0.5pt]{->}(0,3.5)(0.7,3.5)(4.5,2.6)(6,1)
\psecurve[linewidth=0.5pt]{->}(0,3.5)(-0.7,3.5)(-4.5,2.6)(-6,1)
\rput(-1,1.5){\blue $I_+^M(A)$}

\psline[linecolor=red](1.9,0.32)(5,-3)
\psline[linecolor=red](-1.95,0.2)(-3.5,-1.4)
\psline[linecolor=red,linestyle=dashed](-3.5,-1.4)(-5,-3)
\psdots[dotstyle=o](-3.5,-1.4)
\rput(0,-3.5){\red \small $J_-^M(A)$}
\psline[linewidth=0.5pt]{->}(0,-3.2)(0,-2.5)
\psecurve[linewidth=0.5pt]{->}(0,-3.5)(0.7,-3.5)(4.5,-2.6)(6,-1)
\psecurve[linewidth=0.5pt]{->}(0,-3.5)(-0.7,-3.5)(-3,-1)(-6,-1)
\rput(1,-1.5){\red $I_-^M(A)$}

\psclip{\pspolygon[linestyle=none](1.9,0.32)(-1.95,0.2)(-1.95,2)(1.9,2)}
\psccurve[linecolor=red](-2,0)(0,-0.67)(2,0)(1,0.8)(-0.5,0.67)
\endpsclip

\psclip{\pspolygon[linestyle=none](1.92,-0.2)(-1.92,-0.25)(-1.92,-2)(1.92,-2)}
\psccurve[linecolor=blue](-2,0)(0,-0.67)(2,0)(1,0.8)(-0.5,0.67)
\endpsclip

\end{pspicture*} 

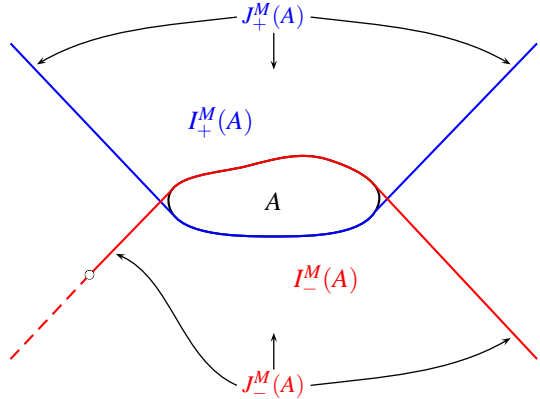
\abb{Causal and chronological future resp.\ past of $A$}

\end{center}

We will also use the notation $J^M(A) := J_-^M(A) \cup J_+^M(A)$.
A subset $A\subset M$ is called {\em past compact}\index{past compact subset}  if $A\cap J_-^M(p)$
is compact for all $p\in M$.
Similarly, one defines {\em future compact} subsets.\index{future compact subset}

\begin{center}

\begin{pspicture*}(-5,-1.8)(4,3)

\psset{unit=0.6cm}

\psclip{\pspolygon[linestyle=none](-5,-2)(-2.5,0.5)(0,-2)}
\pscustom[fillstyle=vlines,hatchwidth=0.6pt,hatchangle=0]{
\psline(-5,-2)(-2.5,0.5)
\psecurve[linecolor=red](-5,10)(-4,4)(-4.5,1)(1,-1)(2,3)(4,4)(5,10)
\psline(-2.5,0.5)(0,-2)}
\endpsclip

\psecurve[linecolor=red](-5,10)(-4,4)(-4.5,1)(1,-1)(2,3)(4,4)(5,10)
\rput(-1,3){\red $A$}

\psdots(-2.5,0.5)
\rput(-2.6,0.8){$p$}

\psline[linecolor=blue](-2.5,0.5)(0,-2)
\psline[linecolor=blue](-2.5,0.5)(-5,-2)
\rput(-2.5,-1.75){\blue $J_-^M(p)$}

\end{pspicture*} 

\abb{Past compact subset}

\end{center}

\Definition{
A subset $S$ of a connected timeoriented Lorentzian manifold is called 
{\em achronal}\index{achronal subset} if  each timelike 
curve meets $S$ in at most one point.
A subset $S$ of a connected timeoriented Lorentzian manifold is called 
\emph{acausal}\index{acausal subset} if  each causal
curve meets $S$ in at most one point.
A subset $S$ of a connected timeoriented Lorentzian manifold is a {\em
Cauchy hypersurface} \index{Cauchy hypersurface} if each inextendible timelike curve in $M$ 
meets $S$ at exactly one point.   
}

\begin{center}

\begin{pspicture*}(0,-3.2)(7,3.2)

\psset{unit=0.6cm}

\psecurve(0,-7)(0.5,-5)(0.5,5)(1,7)

\psecurve(10,-8)(9.5,-5)(9.5,5)(10,8)
\rput(10,2.5){$M$}

\psecurve[linecolor=red](-2,-2)(0.6,-2)(3,-1)(5,-1)
\psecurve[linecolor=red](2,0)(3,-1)(5,-2)(7,-3)
\psecurve[linecolor=red](4,-3)(5,-2)(9.29,-1.5)(11,-2)
\rput(8,-1.8){\red $S$}

\psecurve(0,-7)(0.5,-5)(0.5,5)(1,7)

\psecurve(10,-8)(9.5,-5)(9.5,5)(10,8)

\psecurve[linecolor=blue](3,5)(2,3)(2,-5)(1,-7)
\psdots(2.05,-1.37)

\end{pspicture*} 

\abb{Cauchy hypersurface}

\end{center}

Obviously every acausal subset is achronal, but the reverse is wrong. 
Any Cauchy hypersurface is achronal.
Moreover, it is a closed topological hypersurface and it is hit
by each inextendible causal curve in at least one point.
Any two Cauchy hypersurfaces in $M$ are homeomorphic.
Furthermore, the causal future and past of a Cauchy hypersurface is past 
and future compact respectively.

\Definition{
A Lorentzian manifold is said to satisfy the \emph{causality condition} \index{causality condition}if
it does not contain any closed causal curve.  

A Lorentzian manifold is said to satisfy the {\em strong causality condition}
\index{strong causality condition}if there are no almost closed causal curves.
More precisely, for each point $p\in M$ and for each open neighborhood 
$U$ of $p$ there exists an open neighborhood $V\subset U$ of $p$ such
that each causal curve in $M$ starting and ending in $V$ is entirely
contained in $U$.
}

\begin{center}

\begin{pspicture*}(-1,0)(5,2.5)
        \psdots(3.3,0.8)
        \psccurve(2.5,1.5)(3.9,1.4)(4.4,0.4)(2.4,0.5)
        \pscircle[linecolor=blue,linewidth=0.6pt](3.3,0.8){0.6}
        \psdots[linecolor=red,dotscale=0.75](3.1,0.5)(2.9,0.95)
        \pscurve[linecolor=red,linewidth=0.6pt,border=0.4pt]%
          (3.1,0.5)(1.2,0.2)(0.2,1.4)(2.9,0.95)
        \rput[ l](3.4,0.8){\small$p$}
        \rput[b ](2.8,1.2){\small\blue $V$}
        \rput[bl](4.0,1.4){\small $U$}
        \rput[bl](-1,0.8){\scriptsize\red forbidden!}
\end{pspicture*}

\abb{Strong causality condition}
\end{center}

Obviously, the strong causality condition implies the causality condition.
 
In order to get a good analytical theory for wave operators we must impose
certain geometric conditions on the Lorentzian manifold.
Here are several equivalent formulations.

\begin{thm}\label{globhyp}
Let $M$ be a connected timeoriented Lorentzian manifold.
Then the following are equivalent:
\begin{itemize}
\item[(1)]
$M$ satisfies the strong causality condition and for all
$p,q\in M$ the intersection $J_+^M(p)\cap J_-^M(q)$ is compact.
\item[(2)]
There exists a Cauchy hypersurface in $M$.
\item[(3)]
There exists a smooth spacelike Cauchy hypersurface in $M$.
\item[(4)]
$M$ is foliated by smooth spacelike Cauchy hypersurfaces.
More precisely,
$M$ is isometric to $\R\times S$ with metric $-\beta dt^2 + g_t$ where
$\beta$ is a smooth positive function, $g_t$ is a Riemannian metric
on $S$ depending smoothly on $t\in\R$ and each $\{t\}\times S$ is a smooth 
spacelike Cauchy hypersurface in $M$.
\end{itemize}
\end{thm}

That (1) implies (4) has been shown by Bernal and S\'anchez in 
\cite[Thm.~1.1]{BS} using work of Geroch \cite[Thm.~11]{Geroch}.
See also \cite[Prop.~6.6.8]{EH} and \cite[p.~209]{Wa} for earlier mentionings
of this fact.
The implications $(4) \Rightarrow (3)$ and $(3) \Rightarrow (2)$ are trivial.
That (2) implies (1) is well-known, see e.~g.\ \cite[Cor.~39, p.~422]{ONeill}.

\Definition{
A connected timeoriented Lorentzian manifold satisfying one and hence all
conditions in Theorem~\ref{globhyp} is called {\em globally
hyperbolic}.\index{Lorentzian manifold!globally hyperbolic}
}

\Remark{\label{rem:globhypOmega}
If $M$ is a globally hyperbolic Lorentzian manifold, then a nonempty open
subset $\Omega\subset M$ is itself globally hyperbolic if and only if for any
$p,q\in \Omega$ the intersection $J_+^\Omega(p)\cap
J_-^\Omega(q)\subset\Omega$ is compact.
Indeed non-existence of almost closed causal curves in $M$ directly implies
non-existence of such curves in $\Omega$.
}

\Remark{\label{rem:globhypconf}
It should be noted that global hyperbolicity is a conformal notion.
The definition of a Cauchy hypersurface requires only causal concepts.
Hence if $(M,g)$ is globally hyperbolic and we replace the metric $g$ by a
conformally related metric $\hat g = f\cdot g$, $f$ a smooth positive function
on $M$, then $(M,\hat g)$ is again globally hyperbolic.
}

\Examples{\label{ex:globhyp}
Minkowski space \index{Minkowski space} is globally hyperbolic.
Every spacelike hyperplane is a Cauchy hypersurface.
One can write Minkowski space as $\R\times\R^{n-1}$ with the metric 
$-\dt^2+g_t$ where $g_t$ is the Euclidean metric on $\R^{n-1}$ and does not
depend on $t$.

Let $(S,g_0)$ be a connected Riemannian manifold and $I\subset\R$ an interval.
The manifold $M=I\times S$ with the metric $g=-\dt^2 + g_0$ is
globally hyperbolic if and only if $(S,g_0)$ is complete.
This applies in particular if $S$ is compact.

More generally, if $f:I\to\R$ is a smooth positive function we may equip
$M=I\times S$ with the metric $g=-\dt^2 + f(t)^2\cdot g_0$.
Again, $(M,g)$ is globally hyperbolic if and only if $(S,g_0)$ is complete.
{\em Robertson-Walker spacetimes} \index{Robertson-Walker spacetime}and, in particular, {\em Friedmann
  cosmological models},\index{Friedmann cosmological model}  are of this type. 
They are used to discuss big bang, \index{big bang}expansion of the universe, \index{expansion of the universe}and cosmological
redshift, \index{cosmological
redshift}compare \cite[Ch.~12]{ONeill}.
Another example of this type is {\em deSitter spacetime}, \index{deSitter spacetime}where $I=\R$,
$S=S^{n-1}$, $g_0$ is the canonical metric of $S^{n-1}$ of constant sectional
curvature $1$, and $f(t)=\cosh(t)$.
But {\em Anti-deSitter spacetime} \index{anti-deSitter spacetime}is not globally hyperbolic.

The interior and exterior {\em Schwarzschild spacetimes} \index{Schwarzschild spacetime}are globally
hyperbolic.
They model the universe in the neighborhood of a massive static rotionally
symmetric body such as a black hole.\index{black hole}
They are used to investigate perihelion advance of Mercury, the bending of
light near the sun and other astronomical phenomena, see
\cite[Ch.~13]{ONeill}.
}

\begin{lemma}\label{lJ+KJ-K'cpct}
Let $S$ be a Cauchy hypersurface in a globally hyperbolic Lorentzian manifold
$M$ and let $K,K'\subset M$ be compact.

Then $J_\pm^M(K)\cap S$, $J_\pm^M(K)\cap J_\mp^M(S)$, and $J_+^M(K)\cap
J_-^M(K')$ are compact.
\end{lemma}

\section{Wave operators}

Let ${M}$ be a Lorentzian manifold and let $E\to {M}$ be a real or complex 
vector bundle.
A linear differential operator $P:\,C^{\infty}({M},E)\to 
C^{\infty}({M},E)$ of second order will be called a {\it wave operator} \index{wave operator}or a
{\it normally hyperbolic operator} \index{normally hyperbolic operator} if its principal symbol is given by the
metric, 
$$
\sigma_P(\xi) = -\la\xi,\xi\ra\cdot\id_{E_x}
$$
for all $x\in{M}$ and all $\xi \in T^*_x{M}$.
In other words, if we choose local coordinates $x^1,\ldots,x^n$ on $M$ and a
local trivialization of $E$, then 
$$
P = -\sum_{i,j=1}^n g^{ij}(x)\frac{\partial^2}{\partial x^i \partial x^j}
+ \sum_{j=1}^n A_j(x) \frac{\partial}{\partial x^j} + B(x)
$$
where $A_j$ and $B$ are matrix-valued coefficients depending smoothly on
$x$ and $(g^{ij})_{ij}$ is the inverse matrix of  $(g_{ij})_{ij}$ with
$g_{ij}=\la\frac{\partial}{\partial x^i},\frac{\partial}{\partial x^j}\ra$.

\Example{
Let $E$ be the trivial line bundle so that sections in $E$ are just functions.
The d'Alembert operator \index{d'Alembert operator} $P=\Box=-\div\circ\grad$ is a wave operator.
}

\Example{
Let $E$ be a vector bundle and let $\nabla$ be a connection on $E$.
This connection together with the Levi-Civita connection on $T^*M$ induces a
connection on $T^*M\otimes E$, again denoted $\nabla$.
We define the {\em connection-d'Alembert operator} \index{connection-d'Alembert operator} $\Box^\nabla$ to be minus
the composition of the following three maps
$$
C^\infty(M,E) \stackrel{\nabla}{\longrightarrow}
C^\infty(M,T^*M\otimes E) \stackrel{\nabla}{\longrightarrow}
C^\infty(M,T^*M\otimes T^*M\otimes E)
\xrightarrow{\tr\otimes\id_E} 
C^\infty(M,E)
$$
where $\tr:T^*M\otimes T^*M \to \R$ denotes the metric trace,
$\tr(\xi\otimes\eta)=\la\xi,\eta\ra$. 
We compute the principal symbol,
$$
\sigma_{\Box^\nabla}(\xi)\phi
=
-(\tr\otimes\id_E)\circ\sigma_{\nabla}(\xi)\circ\sigma_{\nabla}(\xi)(\phi)
=
-(\tr\otimes\id_E)(\xi\otimes\xi\otimes\phi)
=
-\la\xi,\xi\ra\,\phi.
$$
Hence $\Box^\nabla$ is a wave operator.
}
\Example{
Let $E=\Lambda^kT^*M$ be the bundle of $k$-forms.
Exterior differentiation $d:C^\infty(M,\Lambda^kT^*M) \to
C^\infty(M,\Lambda^{k+1}T^*M)$ increases the degree by one while
the codifferential $\delta:C^\infty(M,\Lambda^{k}T^*M) \to
C^\infty(M,\Lambda^{k-1}T^*M)$ decreases the degree by one.
While $d$ is independent of the metric, the codifferential $\delta$ does
depend on the Lorentzian metric.
The operator $P=d\delta + \delta d$ is a wave operator.
}
\Example{
If $M$ carries a Lorentzian metric and a spin structure, then one can define
the spinor bundle $\Sigma M$ and the Dirac operator\index{Dirac operator}
$$
D : C^\infty(M,\Sigma M) \to C^\infty(M,\Sigma M) ,
$$
see \cite{Baum} or \cite{BGM} for the definitions.
The principal symbol of $D$ is given by Clifford multiplication,
$$
\sigma_D(\xi)\psi = \xi^\sharp \cdot \psi.
$$
Hence
$$
\sigma_{D^2}(\xi)\psi = \sigma_D(\xi)\sigma_D(\xi)\psi = \xi^\sharp \cdot
\xi^\sharp \cdot \psi = -\la\xi,\xi\ra\,\psi.
$$
Thus $P=D^2$ is a wave operator.
}

\section{The Cauchy problem}

We now come to the basic initial value problem \index{initial value problem}for wave operators, the
{\em Cauchy problem}.\index{Cauchy problem}
The local theory of linear hyperbolic operators can be found in basically any
textbook on partial differential equations.
In \cite{FL2} and \cite{Guenther} the local theory for wave operators on
Lorentzian manifolds is developed.
The results of this section are of global nature.
They make statements about solutions to the Cauchy problem which are defined
globally on a manifold.
Proofs of the results of this section can be found in \cite[Sec.~3.2]{BGP}.

\begin{thm}[Existence and uniqueness of solutions]\label{cauchyglobhyp}\index{Cauchy problem!existence and uniqueness of solutions}
Let $M$ be a globally hyperbolic Lorentzian manifold and let
 $S\subset M$ be a smooth spacelike Cauchy hypersurface.
Let $\nu$ be the future directed timelike unit normal field along $S$.
Let $E$ be a vector bundle over $M$ and let $P$ be a wave
operator acting on sections in $E$.

Then for each  $u_0, u_1 \in \DD(S,E)$ and for each
$f\in\DD(M,E)$ there exists a unique $u\in C^\infty(M,E)$ satisfying
$Pu=f$, $u|_S = u_0$, and $\nabla_\nu u|_S = u_1$.
\end{thm}

It is unclear how to even formulate the Cauchy problem on a Lorentzian
manifold which is not globally hyperbolic.
One would have to replace the concept of a Cauchy hypersurface by something
different to impose the initial conditions upon.
Here are two examples which illustrate what can typically go wrong.

\Example{
Let $M=S^1 \times \R^{n-1}$ with the metric $g=-d\theta^2 + g_0$ where
$d\theta^2$ is the standard metric on $S^1$ of length $1$ and $g_0$ is the
Euclidean metric on $\R^{n-1}$.
The universal covering of $M$ is Minkowski space.

Let us try to impose a Cauchy problem on $\{\theta_0\} \times \R^{n-1}$ which
is the image of a Cauchy hypersurface in Minkowski space.
Such a solution would lift to Minkowski space where it indeed exists uniquely
due to Theorem~\ref{cauchyglobhyp}.
But such a solution on Minkowski space is in general not time periodic, hence
does not descend to a solution on $M$.

Therefore existence of solutions fails.
The problem is here that $M$ violates the causality condition, i.~e.\ there
are closed causal curves.
}

\Remark{
Compact Lorentzian manifolds always possess closed timelike curves and are
therefore never well suited for the analysis of wave operators.
}

\Example{
Let $M$ be a timelike strip in $2$-dimensional Minkowski space, i.~e.\
$M=\R\times (0,1)$ with metric $g=-dt^2 + dx^2$.
Let $S:= \{0\}\times(0,1)$.
Given any $u_0, u_1 \in \DD(S,E)$ and any $f\in\DD(M,E)$, there exists a
solution $u$ to the Cauchy problem.
One can simply take the solution in Minkowski space and restrict it to $M$.
But this solution is not unique in $M$.
Choose $x$ in Minkowski space, $x\not\in M$, such that
$J_+^{\mathrm{Mink}}(x)$ intersects $M$ in the future of $S$ and of $\supp(f)$.
The advanced fundamental solution $w=F_+(x)$ (see next section) has support
contained in $J_+^{\mathrm{Mink}}(x)$ and satisfies $Pw=0$ away from $x$.
Hence $u+w$ restricted to $M$ is again a solution to the Cauchy problem on $M$
with the same initial data.

\begin{center}

\begin{pspicture*}(-5.5,-1.8)(6,3.5)

\psset{unit=0.8cm}

\psline[linecolor=blue,linestyle=dashed](0,-2)(0,4)
\psline[linecolor=blue,linestyle=dashed](2,-2)(2,4)
\rput(1,-1.5){\blue $M$}

\psline[linewidth=1.5pt](0,0)(2,0)
\rput(2.3,0){$S$}

\psellipse[linecolor=red](1.3,1)(0.5,0.6)
\rput(3.5,1){\red $\supp(f)$}
\psecurve[linecolor=red]{->}(3.5,0.5)(2.5,1)(1.3,1)(0.3,0.5)

\qdisk(-1,1){1.5pt}
\rput(-1,0.7){$x$}
\rput(-1.5,3){$J^{\mathrm{Mink}}_+(x)$}
\psline[linestyle=dashed](-1,1)(2,4)
\psline[linestyle=dashed](-1,1)(-4,4)

\psline[linewidth=1.5pt,linecolor=red](0.5,0)(0.8,0)
\rput(4.6,-1){\red $\supp(u_0)\cup\supp(u_1)$}
\psline[linestyle=dotted,linecolor=red,dotsep=1pt](0.5,0)(0,0.5)
\psline[linestyle=dotted,linecolor=red,dotsep=1pt](0.8,0)(1.2,0.4)
\psline[linestyle=dotted,linecolor=red,dotsep=1pt](1.7,0.6)(2,0.95)
\psecurve[linecolor=red]{->}(3.5,-0.5)(2.5,-1)(0.65,-0.1)(0.65,0.5)

\end{pspicture*} 

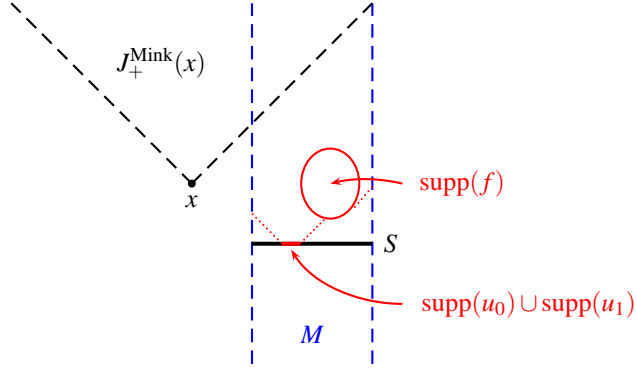
\abb{Nonunique solution to Cauchy problem}

\end{center}

The problem is here that $S$ is acausal but not a Cauchy hypersurface.
Physically, a wave ``from outside the manifold'' enters into $M$.
}

The physical statement that a wave can never propagate faster than with the
speed of light is contained in the following.

\begin{thm}[Finite propagation speed]\index{finite propagation speed}
The solution $u$ from Theorem~\ref{cauchyglobhyp} satisfies $\supp(u) \subset
J^M(K)$ where $K=\supp(u_0)\cup \supp(u_1) \cup \supp(f)$.
\end{thm}

The solution to the Cauchy problem depends continuously on the data.

\begin{thm}[Stability]\label{cauchystetig}\index{Cauchy problem!stability of solutions}
Let $M$ be a globally hyperbolic Lorentzian manifold and let 
$S\subset M$ be a smooth spacelike Cauchy hypersurface.
Let $\nu$ be the future directed timelike unit normal field along $S$.
Let $E$ be a vector bundle over $M$ and let $P$ be a wave
operator acting on sections in $E$.

Then the map $\DD(M,E) \oplus \DD(S,E) \oplus \DD(S,E) \to C^\infty(M,E)$
sending $(f,u_0,u_1)$ to the unique solution $u$ of the Cauchy
problem $Pu=f$, $u|_S=u|_0$, $\nabla_\nu u = u_1$ is linear continuous.
\end{thm}

This is essentially an application of the open mapping theorem for Fr\'echet
spaces.

\section{Fundamental solutions}

\Definition{\label{deffundsol}
Let $M$ be a timeoriented Lorentzian manifold, let $E\to M$ be a
vector bundle and let $P:\,C^{\infty}(M,E)\to C^{\infty}(M,E)$ be
a wave operator. 
Let $x\in M$.
A {\em fundamental solution} \index{fundamental solution}of $P$ at $x$ is a distribution 
$F\in\DD'(M,E,E_x^*)$ such that
$$
PF=\delta_x.
$$
In other words, for all $\varphi\in\DD(M,E^*)$ we have
$$
F[P^*\varphi] = \varphi(x).
$$
If $\supp(F(x))\subset J_+^M(x)$, then we call $F$ an {\em advanced
fundamental solution}, \index{fundamental solution!advanced}if $\supp(F(x))\subset J_-^M(x)$, then we call $F$
a {\em retarded fundamental solution}.\index{fundamental solution!retarded}
}

Using the knowlegde about the Cauchy problem from the
previous section it is now not hard to find global fundamental solutions 
on a globally hyperbolic manifold.

\begin{thm}\label{globhypexist}
Let $M$ be a globally hyperbolic Lorentzian manifold. 
Let $P$ be a wave operator acting on sections in a 
vector bundle $E$ over $M$.

Then for every $x\in M$ there is exactly one
fundamental solution $F_+(x)$ for $P$ at $x$ with past compact
support and exactly one fundamental solution $F_-(x)$ for $P$ at $x$ with 
future compact support.
They satisfy
\begin{enumerate}
\item
$\supp(F_\pm(x))\subset J_\pm^M(x)$,
\item
for each $\varphi\in\DD(M,E^*)$ the maps $x\mapsto F_\pm(x)[\varphi]$
are smooth sections in $E^*$ satisfying the differential equation
$P^*(F_\pm(\cdot)[\varphi])=\varphi$.
\end{enumerate}
\end{thm}

\begin{proof}[Sketch of proof.]
We do not do the uniqueness part.
To show existence fix a foliation of $M$ by spacelike Cauchy hypersurfaces
$S_t$, $t\in\R$ as in Theorem~\ref{globhyp}.
Let $\nu$ be the future directed unit normal field along the leaves $S_t$.
Let $\varphi\in\DD(M,E^*)$.
Choose $t$ so large that $\supp(\varphi)\subset I_-^M(S_t)$.
By Theorem~\ref{cauchyglobhyp} there exists a unique $\chi_\varphi\in
C^\infty(M,E^*)$ such that $P^*\chi_\varphi=\varphi$ and $\chi_\varphi|_{S_t}
= (\nabla_\nu\chi_\varphi)|_{S_t} = 0$.
One can check that $\chi_\varphi$ does not depend on the choice of $t$.

Fix $x\in M$.
By Theorem~\ref{cauchystetig} $\chi_\varphi$ depends continuously on
$\varphi$. 
Since the evaluation map $C^\infty(M,E) \to E_x$ is continuous, the
map $\DD(M,E^*) \to E_x^*$, $\varphi \mapsto \chi_\varphi(x)$, is also
continuous.
Thus $F_+(x)[\varphi] := \chi_\varphi(x)$ defines a distribution.
By definition $P^*(F_+(\cdot)[\varphi]) = P^*\chi_\varphi = \varphi$.

Now $P^*\chi_{P^*\varphi} = P^*\varphi$, hence
$P^*(\chi_{P^*\varphi}-\varphi)=0$.
Since both $\chi_{P^*\varphi}$ and $\varphi$ vanish along $S_t$ the uniqueness
part which we have omitted shows $\chi_{P^*\varphi}=\varphi$.
Thus
$$
(PF_+(x))[\varphi] = F_+(x)[P^*\varphi] = \chi_{P^*\varphi}(x) = \varphi(x)
= \delta_x[\varphi].
$$
Hence $F_+(x)$ is a fundamental solution of $P$ at $x$.

It remains to show $\supp(F_+(x)) \subset J_+^M(x)$.
Let $y\in M\setminus J_+^M(x)$.
We have to construct a neighborhood of $y$ such that for each test section
$\varphi\in\DD(M,E^*)$ whose support is contained in this neighborhood
we have $F_+(x)[\varphi] = \chi_\varphi(x) = 0$.
Since $M$ is globally hyperbolic $J_+^M(x)$ is closed and therefore $J_+^M(x)\cap
J_-^M(y') = \emptyset$ for all $y'$ sufficiently close to $y$.
We choose $y'\in I_+^M(y)$ and $y''\in I_-^M(y)$ so close that $J_+^M(x)\cap
J_-^M(y') = \emptyset$ and $\left(J_+^M(y'') \cap\bigcup_{t \leq t'}S_t\right)\cap
J_+^M(x) = \emptyset$ where $t'\in\R$ is such that $y'\in S_{t'}$. 

\begin{center}

\begin{pspicture*}(-5,-5)(5,3)

\psset{unit=0.9cm}
\psecurve(-5,-7)(-4.5,-5)(-4.5,5)(-4,7)

\psecurve(5,-8)(4.5,-5)(4.5,5)(5,8)


\psecurve[linecolor=red](-6,1)(-4.5,0)(4.27,0)(7,-1)
\rput(4.7,0){\red$S_{t'}$}

\psdots(-2.75,-0.9)
\rput(-2.85,-0.75){$y$}

\psdots(-3,-0.27)
\rput(-3,0.1){$y'$}

\psecurve(-4,0.7)(-3,-0.27)(3,-6.27)(4,-7.27)
\psecurve(-2,0.7)(-3,-0.27)(-4.4,-1.77)(-6,-3.27)
\rput(-1,-4){$J_-^M(y')$}

\psdots(-2.5,-1.5)
\rput(-2.5,-1.9){$y''$}

\psecurve[linecolor=blue](-1.5,-2.5)(-2.5,-1.5)(-3.83,-0.17)(-5,1)
\psecurve[linecolor=blue](-3.5,-2.5)(-2.5,-1.5)(-1.17,-0.23)(0,1)
\rput(-2.9,1.5){\small\blue $J_+^M(y'')\cap(\cup_{t\leq t'}S_t)$}
\psecurve[linecolor=blue,linewidth=0.5pt]{->}(-2.5,1.5)(-2.5,1.2)(-2,-0.7)(0,-3)

\psdots(2,-1.5)
\rput(2,-1.8){$x$}

\psecurve(3,-2.5)(2,-1.5)(-3,3.5)(-5,5.5)
\psecurve(1,-2.5)(2,-1.5)(4.3,0.8)(6,2.5)
\rput(2,1.5){$J_+^M(x)$}


\psdots(-2.5,-1.5)

\psecurve(-5,-7)(-4.5,-5)(-4.5,5)(-4,7)

\psecurve(5,-8)(4.5,-5)(4.5,5)(5,8)

        \end{pspicture*} 

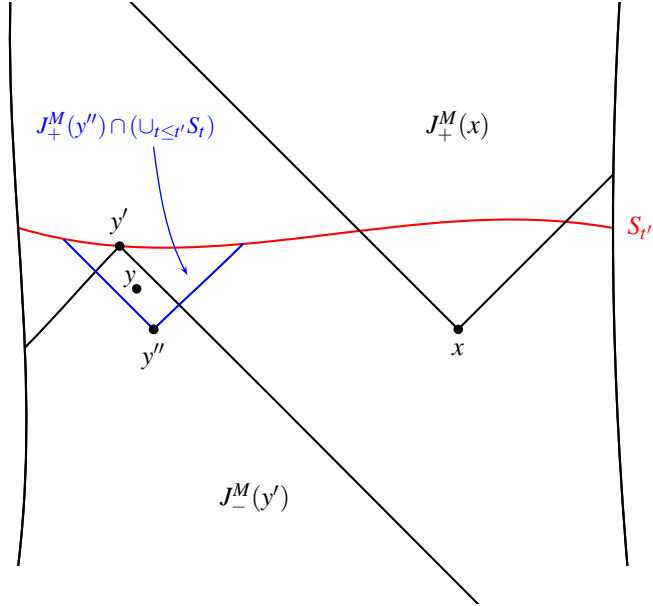
\abb{Construction of $y$, $y'$ and $y''$}

\end{center}

Now $K:=J_-^M(y')\cap J_+^M(y'')$ is a compact neighborhood of $y$.
Let $\varphi\in\DD(M,E^*)$ be such that $\supp(\varphi)\subset K$.
By Theorem~\ref{cauchyglobhyp} $\supp(\chi_\varphi) \subset J_+^M(K)\cup J_-^M(K)
\subset J_+^M(y'') \cup J_-^M(y')$.
By the independence of $\chi_\varphi$ of the choice of $t>t'$ we have that
$\chi_\varphi$ vanishes on $\bigcup_{t> t'}S_t$.
Hence $\supp(\chi_\varphi) \subset \left(J_+^M(y'') \cap \bigcup_{t \leq
t'}S_t\right) \cup J_-^M(y')$ and is therefore disjoint from $J_+^M(x)$.
Thus $F_+(x)[\varphi] = \chi_\varphi(x) = 0$ as required.
\end{proof}

For a complete proof see \cite[Sec.~3.3]{BGP}.

\section{Green's operators}

Now we want to find ``solution operators'' for a given wave
operator $P$.
More precisely, we want to find operators which are inverses of $P$ when
restricted to suitable spaces of sections.
We will see that existence of such operators is basically equivalent to the
existence of fundamental solutions.

\Definition{ \label{defGfunc}
Let $M$ be a timeoriented connected Lorentzian manifold.
Let $P$ be a wave operator acting on sections in a  
vector bundle $E$ over $M$.
A linear map $G_+ : \DD(M,E) \to C^\infty(M,E)$ satisfying
\begin{itemize}
\item[(i)]
$P\circ G_+ = \id_{\DD(M,E)}$,
\item[(ii)]
$G_+ \circ P|_{\DD(M,E)} = \id_{\DD(M,E)}$,
\item[(iii)]
$\supp(G_+\varphi) \subset J_+^M(\supp(\varphi))$ for all
$\varphi\in\DD(M,E)$,
\end{itemize}
is called an {\em advanced Green's operator for} \index{Green's operator!advanced}$P$.
Similarly, a linear map $G_- : \DD(M,E) \to C^\infty(M,E)$ satisfying
(i), (ii), and
\begin{itemize}
\item[(iii')]
$\supp(G_-\varphi) \subset J_-^M(\supp(\varphi))$ for all
$\varphi\in\DD(M,E)$
\end{itemize}
instead of (iii) is called a {\em retarded Green's operator for} $P$.\index{Green's operator!retarded}
}

Fundamental solutions and Green's operators are closely related.

\begin{thm}\label{functorsolve}
Let $M$ be a globally hyperbolic Lorentzian manifold.
Let $P$ be a wave operator acting on sections in a  
vector bundle $E$ over $M$.

Then there exist unique advanced and retarded Green's operators $G_\pm: 
\DD(M,E) \to C^\infty(M,E)$ for $P$.
\end{thm}

\begin{proof}
By Theorem~\ref{globhypexist} there exist families $F_\pm(x)$ of advanced and
retarded fundamental solutions for the adjoint operator $P^*$ respectively.
We know that $F_\pm(x)$ depend smoothly on $x$ and the
differential equation $P(F_\pm(\cdot)[\varphi])=\varphi$ holds.
By definition we have
$$
P(G_\pm\varphi) = P(F_\mp(\cdot)[\varphi]) = \varphi
$$
thus showing (i).
Assertion (ii) follows from the fact that the $F_\pm(x)$ are fundamental 
solutions,
$$
G_\pm(P\varphi)(x) = F_\mp(x)[P\varphi] = P^*F_\mp(x)[\varphi]
= \delta_x[\varphi] = \varphi(x).
$$
To show (iii) let $x\in M$ such that $(G_+\varphi)(x) \not= 0$.
Since $\supp(F_-(x))\subset J_-^M(x)$ the support of $\varphi$ must hit
$J_-^M(x)$.
Hence $x\in J_+^M(\supp(\varphi))$ and therefore $\supp(G_+\varphi)\subset
J_+^M(\supp(\varphi))$.
The argument for $G_-$ is analogous.
\end{proof}

We have seen that existence of fundamental solutions for $P^*$ depending
nicely on $x$ implies existence of Green's operators for $P$.
This construction can be reversed.
Then uniqueness of fundamental solutions in Theorem~\ref{globhypexist}
implies uniqueness of Green's operators.

\begin{lemma}\label{Greenadjungiert}
Let $M$ be a globally hyperbolic Lorentzian manifold.
Let $P$ be a wave operator acting on sections in a  
vector bundle $E$ over $M$.
Let $G_\pm$ be the Green's operators for $P$ and $G^*_\pm$ the Green's
operators for the adjoint operator $P^*$.
Then
\begin{equation}
\int_M (G^*_\pm\varphi)\cdot\psi \dV = \int_M \varphi\cdot(G_\mp\psi) \dV
\label{Gsa}
\end{equation}
holds for all $\varphi\in\DD(M,E^*)$ and $\psi\in\DD(M,E)$.
\end{lemma}

\begin{proof}
For the Green's operators we have $PG_\pm=\id_{\DD(M,E)}$ and
$P^*G_\pm^*=\id_{\DD(M,E^*)}$ and hence
\begin{eqnarray*} 
\int_M (G_\pm^*\varphi)\cdot\psi \dV 
&=& \int_M (G_\pm^*\varphi)\cdot(PG_\mp\psi) \dV  \\
&=& \int_M (P^*G_\pm^*\varphi)\cdot(G_\mp\psi) \dV \\
&=& \int_M \varphi\cdot(G_\mp\psi) \dV .
\end{eqnarray*}
Notice that $\supp(G_\pm\phi) \cap \supp(G_\mp\psi)\subset
J_\pm^M(\supp(\phi))\cap J_\mp^M(\supp(\psi))$ is compact in a globally hyperbolic
manifold so that the partial integration in the second equation is justified.
\end{proof}

\Notation{
We write $\Csc(M,E)$ for the set of all $\phi\in C^\infty(M,E)$ for which
there exists a compact subset $K \subset M$ such that $\supp(\phi) \subset
J^M(K)$.
Obviously, $\Csc(M,E)$ is a vector subspace of $C^\infty(M,E)$.

The subscript ``sc'' should remind the reader of ``space-like compact''.
Namely, if $M$ is globally hyperbolic and $\phi\in\Csc(M,E)$, then for every
Cauchy hypersurface $S\subset M$ the support of $\phi|_{S}$ is contained in 
$S\cap J^M(K)$ hence compact by Lemma~\ref{lJ+KJ-K'cpct}.
In this sense sections in $\Csc(M,E)$ have space-like compact support.\index{space-like compact support}
}

\Definition{
We say that a sequence of elements $\phi_j\in \Csc(M,E)$ {\em converges in
 $\Csc(M,E)$ to} $\phi\in \Csc(M,E)$ if there exists a compact subset $K
\subset M$ such that 
$$
\supp(\phi) \subset J^M(K)
\mbox{ and }
\supp(\phi_j) \subset J^M(K)
$$ 
for all $j$ and 
$$
\|\phi_j - \phi\|_{C^k(K',E)}\to 0
$$
for all $k\in\N$ and all compact subsets $K'\subset M$.
}

If $G_+$ and $G_-$ are advanced and retarded Green's operators for $P$
respectively, then we get a linear map
$$
G := G_+ - G_- : \DD(M,E) \to \Csc(M,E).
$$

Much of the solution theory of wave operators on globally hyperbolic
Lorentzian manifolds is collected in the following theorem.

\begin{thm}\label{thmExSeq}
Let $M$ be a globally hyperbolic Lorentzian manifold.
Let $P$ be a wave operator acting on sections in a  
vector bundle $E$ over $M$.
Let $G_+$ and $G_-$ be advanced and retarded Green's operators for $P$ respectively.

Then
\begin{equation}
0 \to \DD(M,E) \stackrel{P}{\longrightarrow} \DD(M,E) 
\stackrel{G}{\longrightarrow} \Csc(M,E) \stackrel{P}{\longrightarrow}\Csc(M,E)
\label{eqExSeq}
\end{equation}
is an exact sequence of linear maps.
\end{thm}

\begin{proof}
Properties (i) and (ii) in Definition~\ref{defGfunc} of Green's operators
directly yield $G\circ P=0$ and $P\circ G=0$, both on $\DD(M,E)$.
Properties (iii) and (iii') ensure that $G$ maps $\DD(M,E)$ to $\Csc(M,E)$.
Hence the sequence of linear maps forms a complex.

Exactness at the first $\DD(M,E)$ means that 
$$
P : \DD(M,E) \to \DD(M,E)
$$
is injective.
To see injectivity let $\phi\in\DD(M,E)$ with $P\phi=0$.
Then $\phi= G_+P\phi=G_+0=0$.

Next let $\phi\in\DD(M,E)$ with $G\phi=0$, i.~e.\ $G_+\varphi=G_-\varphi$. 
We put $\psi:=G_+\varphi=G_-\varphi \in C^\infty(M,E)$ and we see
$\supp(\psi)=\supp(G_+\phi)\cap\supp(G_-\phi)\subset J_+^M(\supp(\varphi))\cap
J_-^M(\supp(\varphi))$.
Since $(M,g)$ is globally hyperbolic $J_+^M(\supp(\varphi))\cap
J_-^M(\supp(\varphi))$ is compact, hence $\psi\in \DD(M,E)$.
From $P(\psi) = P(G_+(\phi)) = \phi$ we see that $\phi\in P(\DD(M,E))$.
This shows exactness at the second $\DD(M,E)$.

Finally, let $\phi\in\Csc(M,E)$ such that $P\phi=0$. 
Without loss of generality we may assume that $\supp(\phi)\subset I_+^M(K)\cup I_-^M(K)$ for
a compact subset $K$ of $M$. 
Using a partition of unity subordinated to the open covering $\{I_+^M(K),I_-^M(K)\}$ write $\phi$ as $\phi=\phi_1 + \phi_2$
where $\supp(\phi_1)\subset I_-^M(K)\subset J_-^M(K)$ and $\supp(\phi_2)\subset I_+^M(K)\subset J_+^M(K)$.
For $\psi:=-P\phi_1=P\phi_2$ we see that $\supp(\psi) \subset J_-^M(K) \cap
J_+^M(K)$, hence $\psi\in\DD(M,E)$.

We check that $G_+\psi=\phi_2$.
For all $\chi\in\DD(M,E^*)$ we have
$$
\int_M \chi\cdot(G_+P\phi_2) \dV =
\int_M (G_-^{*}\chi)\cdot(P\phi_2) \dV =
\int_M (P^*G_-^{*}\chi)\cdot\phi_2 \dV =
\int_M \chi\cdot\phi_2 \dV 
$$
where $G_-^*$ is the Green's operator for the adjoint operator $P^*$ according
to Lemma~\ref{Greenadjungiert}.
Notice that for the second equation we use the fact that $\supp(\phi_2) \cap
\supp(G^*_-\chi) \subset J^M_+(K)\cap J^M_-(\supp(\chi))$ is compact.
Similarly, one shows $G_-\psi=-\phi_1$.

Now $G\psi = G_+\psi - G_-\psi = \phi_2 + \phi_1 = \phi$, hence $\phi$ is in
the image of $G$. 
\end{proof}

\begin{prop}\label{PGstetig}\index{Green's operator!continuity of} 
Let $M$ be a globally hyperbolic Lorentzian manifold, 
let $P$ be a wave operator acting on sections in a  
vector bundle $E$ over $M$.
Let $G_+$ and $G_-$ be the advanced and retarded Green's operators for $P$
respectively.

Then $G_\pm:\DD(M,E) \to \Csc(M,E)$ and all maps in the complex
(\ref{eqExSeq}) are continuous. 
\end{prop}

\begin{proof}
The maps $P:\DD(M,E) \to \DD(M,E)$ and $P:\Csc(M,E) \to \Csc(M,E)$ are
continuous simply because $P$ is a differential operator.
It remains to show that $G:\DD(M,E) \to \Csc(M,E)$ is continuous.

Let $\phi_j,\phi\in\DD(M,E)$ and $\phi_j\to \phi$ in $\DD(M,E)$ for all $j$.
Then there exists a compact subset $K\subset M$ such that
$\supp(\phi_j) \subset K$ for all $j$ and $\supp(\phi) \subset K$.
Hence $\supp(G\phi_j)\subset J^M(K)$ for all $j$ and $\supp(G\phi)\subset J^M(K)$.
From the proof of Theorem~\ref{globhypexist} we know that $G_+\phi$ coincides
with  the solution $u$ to the Cauchy problem $Pu=\phi$ with initial conditions
$u|_{S_-} = (\nabla_\nu u)|_{S_-}=0$ where $S_- \subset M$ is a spacelike
Cauchy hypersurface such that $K \subset I_+^M(S_-)$.
Theorem~\ref{cauchystetig}  tells us that if $\phi_j\to\phi$ in $\DD(M,E)$, 
then the solutions $G_+\phi_j\to G_+\phi$ in $C^\infty(M,E)$.
The proof for $G_-$ is analogous and the statement for $G$ follows.
\end{proof}



\printindex

\end{document}